\documentclass{amsart}
\usepackage[english]{babel}
\usepackage[utf8]{inputenc}
\usepackage{amsmath, amssymb,epic,graphicx,mathrsfs,enumerate}
\usepackage[all]{xy}
\usepackage{color}
\usepackage{comment}
\usepackage{enumitem}
\usepackage{hyperref}
\usepackage[]{frontespizio}
\usepackage{booktabs}
\usepackage{array}

\usepackage{amsmath}
\usepackage{amsthm}
\usepackage{amssymb}
\usepackage{latexsym}
\usepackage{epsfig}

\DeclareMathOperator{\perm}{Sym}

\DeclareMathOperator{\aut}{Aut} 
\DeclareMathOperator{\inn}{Inn} 
 
\DeclareMathOperator{\soc}{soc}

\DeclareMathOperator{\psl}{PSL}

\DeclareMathOperator{\psp}{PSp}

\DeclareMathOperator{\M}{M}

\DeclareMathOperator{\psu}{PSU}

\newcommand{\alt}{\mathrm{Alt}}

\newtheorem{thm}{Theorem}

 \newtheorem{lemma}[thm]{Lemma}
\newtheorem{prop}[thm]{Proposition} 
 \newtheorem{defn}[thm]{Definition}

\numberwithin{equation}{section}

\renewcommand{\footnote}{\endnote}
\newcommand{\ignore}[1]{}\makeglossary

\begin{document}
	\bibliographystyle{amsplain}
\title[The nilpotency probability]{Characterizing finite solvable groups through the nilpotency probability}

\author{Andrea Lucchini}
\address{Andrea Lucchini\\ Universit\`a di Padova\\  Dipartimento di Matematica \lq\lq Tullio Levi-Civita\rq\rq\\ Via Trieste 63, 35121 Padova, Italy\\email: lucchini@math.unipd.it}

\begin{abstract}
Given a finite group $G$, we denote by $\nu(G)$ the probability that two randomly chosen elements of $G$ generate a nilpotent subgroup. We prove that if $\nu(G)>1/12,$ then $G$ is solvable.
\end{abstract}

\maketitle

\section{Introduction}

Given a finite group $G$, let $\mathcal N(G)$ be the set of ordered pairs $(g_1,g_2)\in G^2$ with the property that
$\langle g_1, g_2\rangle$ is a nilpotent subgroup of $G$. Then
$$\nu(G):=\frac{|\mathcal N(G)|}{|G|^2}$$
expresses the probability that two randomly chosen elements of $G$ generate a nilpotent subgroup.
Guralnick and Wilson \cite{gw} proved that if $\nu(G) > \nu(\perm(3)) = 1/2$, then $G$ is nilpotent. However, the following question, raised in \cite{cgk}, remains open: does $\nu(G) > 1/12$ imply that $G$ is solvable? In this note, we provide an affirmative answer.


\begin{thm}\label{main}
	Let $G$ be a finite group. If $\nu(G)>\frac{1}{12},$ then
	$G$ is solvable.
\end{thm}
Since $\nu(\alt(5))=1/12,$ the previous result is sharp.

\section{Proof of Theorem \ref{main}}

Let us start this section with a couple of definitions.

\begin{defn}Let $G$ be a finite group, $N$ a normal subgroup of $G$ and $g_1, g_2\in G.$ We define
	$$\begin{aligned}\mathcal N_{g_1,g_2}(G,N)&:=\{(n_1,n_2)\in N^2\mid\langle n_1g_1, n_2g_2\rangle \text { is nilpotent }\},\\\nu_{g_1,g_2}(G,N)&:=\frac{|\mathcal N_{g_1,g_2}(G,N)|}{|N|^2}.
		\end{aligned}$$
\end{defn}
\begin{defn}
	Let $S$ be a finite nonabelian simple group.
	We identify $S$ with the subgroup of its inner automorhisms and we define
	$$\tilde\nu(S)=\max_{a_1,a_2\in \aut S}\nu_{a_1,a_2}(\aut S, S).$$
\end{defn}
\begin{prop}\label{mono}
	Assume that a finite group $G$ contains a unique minimal normal subgroup, say $N$, and that $N$ is nonabelian. Then
	$$\nu(G)\leq \tilde\nu(S),$$
	where $S$ is a composition factor of $N.$
\end{prop}
\begin{proof}
Assume that $N=S^u,$ with $u\in \mathbb N.$ 	We may identify $G$ with a subgroup of the wreath product $\aut S\wr \perm(u).$  In this identification, any element of $G$ can be written in the form $(\alpha_1,\dots,\alpha_u)\rho,$ with $\alpha_1,\dots,\alpha_u \in \aut S$ and $\rho \in \perm(u).$ 
Moreover, we identify $S$ with the inner automorphism group $\inn(S)$ and therefore
$N=\{(s_1,\dots,s_u)\mid s_1,\dots,s_u \in S\}.$

Now let $$g_1=(a_1,\dots,a_u)\sigma,\quad g_2=(b_1,\dots,b_u)\tau \quad \in G.$$
Write $$\sigma=\sigma_1\cdots \sigma_c, \quad  \tau=\tau_1\cdots \tau_d,$$ as products of disjoint cyclic permutations and assume that 1 belongs to the supports of $\sigma_1$ and $\tau_1.$ Suppose
$$\sigma_1=(1,i_2,\dots,i_e), \quad \tau_1=(1,j_2,\dots,j_f).$$
Let $$X=\{(\alpha_1,\dots,\alpha_u)\rho \in G \mid \rho(1)=1\}.$$ The map $\phi\colon X\to \aut S$ sending $(\alpha_1,\dots,\alpha_u)\rho$ to $\alpha_1$ is a group homomorphism. Now let $n_1=(s_1,\dots,s_u), n_2=(t_1,\dots,t_u) \in N$. Then
$$(n_1g_1)^e, (n_2g_2)^f \in X$$ and
$$\phi((n_1g_1)^e)=s_1a_1s_{i_2}a_{i_2}\cdots s_{i_e}a_{i_e},\quad
\phi((n_2g_2)^f)=t_1b_1t_{j_2}b_{j_2}\cdots t_{j_f}b_{j_f}.
$$
Let
$$\gamma_1=a_{1}s_{i_2}a_{i_2}\cdots s_{i_e}a_{i_e},\quad \gamma_2=b_{1}t_{j_2}b_{j_2}\cdots t_{j_f}b_{j_f}.$$
If $(s_1,t_1)\notin  \mathcal N_{\gamma_1,\gamma_2}(\aut S,S),$ then
	$$\phi(\langle (n_1g_1)^{e}, (n_2g_2)^f\rangle)=
	\langle s_1\gamma_1,t_1\gamma_2\rangle$$ is not nilpotent. In particular $\langle n_1g_1, n_2g_2\rangle$ is not nilpotent. Thus, given $s_i, t_j$ for $i\neq 1$ and $j\neq 1$, there are at least $$|S|^2(1-\nu_{\gamma_1,\gamma_2}(\aut S, S))\geq |S|^2(1-\tilde\nu(S))$$ choices for 
	$s_1$ and $t_1$ such that $\langle n_1g_1, n_2g_2\rangle$ is not nilpotent. It follows that there
	are at least $|S|^{2u}(1-\tilde \nu(S))$ pairs $(n_1,n_2)\in N^2$ such that $\langle n_1g_1,n_2g_2\rangle$ is not nilpotent. This is equivalent to saying that $\nu_{g_1,g_2}(G,N)\leq \tilde \nu(S)$ for every $(g_1,g_2) \in G^2.$ But then
	$$\begin{aligned}\nu(G)|G|^2=&\sum_{(Ng_1,Ng_2)\in  (G/N)^2}\nu_{g_1,g_2}(G,N)\\&\leq \sum_{(Ng_1,Ng_2)\in (G/N)^2}|N|^2\tilde\nu(S)\\&\leq |G|^2\tilde\nu(S),
	\end{aligned}$$
	and consequently $\nu(G)\leq \tilde \nu(S).$
\end{proof}
\begin{prop}\label{simple}
	If $S$ is a finite nonabelian simple group, then 
	$$\tilde \nu(S)\leq \tilde \nu(\alt(5))=\frac{1}{12},$$ with equality only if $S=\alt(5).$
\end{prop}
\begin{proof}
	Given $a_1,a_2 \in \aut S$, let $\mathcal P_{a_1,a_2}(S)=\{(s_1,s_2)\in S^2\mid S\leq \langle s_1a_1,s_2a_2\rangle\}$. We define
		$$\pi_{a_1,a_2}(S):=\frac{|\mathcal P_{a_1,a_2}(S)|}{|S|^2}, \quad
\tilde \pi(S)=\min_{a_1,a_2\in \aut S}\pi_{a_1,a_2}(S).$$
Clearly if $(s_1,s_2) \in \mathcal P_{a_1,a_2}(S),$ then
$\langle s_1a_1, s_2a_2\rangle$ is not nilpotent, and therefore $(s_1,s_2) \notin \mathcal N_{a_1,a_2}(\aut S,S).$
Thus, $$\pi_{a_1,a_2}(S)\leq 1-\nu_{a_1,a_2}(\aut(S),S)$$ for every $a_1, a_2\in \aut S,$ and therefore $$\tilde \nu(S)\leq 1-\tilde\pi(S).$$
Results of Dixon \cite{dix}, Kantor–Lubotzky \cite{kl} and Liebeck–Shalev [ls] establish
that $\tilde\pi(S) \to 1$ as $|S| \to \infty.$ 
Explicit lower bounds for $\tilde\pi(S)$ are given in \cite{nina}, where in particular all nonabelian simple groups $S$ with $\tilde\pi(S)\leq 9/10$ are listed. Since $1-9/10=1/10>1/12$, this is not enough for our purposes, but  more detailed information can be found in \cite{tmenezes}. In particular, combining \cite[Table 5.1, Theorem 6.01, Lemma 7.1.1]{tmenezes}, it turns out that if $S$ is not an alternating group, then either $\tilde \pi(S)> \frac{11}{12}$ (and consequently $\tilde\nu(S) < \frac{1}{12}),$ or
$S$ belongs to the family $\mathcal S$  of the following groups:  $$\begin{aligned}&\psl(2,7), \psl(2,8), \psl(2,11), \psl(2,13), \psl(3,3),\\& \psl(3,4), \psp(6,2), \psu(4,2), \mathrm{M}_{11},\M_{12}.\end{aligned}$$ The GAP Library of Tables of Marks \cite{gap} contains the tables of marks of all almost simple groups whose socle belongs to $\mathcal S$, and using them the exact value of $\tilde{\nu}(S)$, reported in Table \ref{tabnu}, can be easily obtained (see Section 3 for more details).
\begin{table}[h]
	\centering
	\renewcommand{\arraystretch}{1.4}
	\begin{tabular}{cc}
		\toprule
		\textbf{$S$} & \textbf{$\tilde{\nu}(S)$} \\
		\midrule
		$\mathrm{PSL}(2,7)$  & $\tfrac{3}{56}    \approx 0.0536$ \\
		$\mathrm{PSL}(2,8)$  & $\tfrac{1}{56}    \approx 0.0179$ \\
		$\mathrm{PSL}(2,11)$ & $\tfrac{2}{165}   \approx 0.0121$ \\
		$\mathrm{PSL}(2,13)$ & $\tfrac{3}{364}   \approx 0.0082$ \\
		$\mathrm{PSL}(3,3)$  & $\tfrac{1}{234}   \approx 0.0043$ \\
		$\mathrm{PSL}(3,4)$  & $\tfrac{5}{4032}  \approx 0.0032$ \\
		$\mathrm{PSU}(4,2)$  & $\tfrac{67}{23760} \approx 0.0026$ \\
		$\mathrm{PSp}(6,2)$  & $\tfrac{1}{4536}  \approx 0.0002$ \\
		$\M_{11}$   & $\tfrac{1}{440}   \approx 0.0023$ \\
		$\M_{12}$   & $\tfrac{7}{11880} \approx 0.0006$ \\
		\bottomrule
	\end{tabular}
	\caption{$\tilde\nu(S)$ for $S\in \mathcal S$}
	\label{tabnu}
\end{table}

Now assume that $S=\alt(n)$. Using GAP, and in particular the library of table marks, it can be checked 
that
$$\tilde\nu(\alt(5))=\tfrac{1}{12}, \
\tilde\nu(\alt(6))=\tfrac{1}{36}, \
\tilde\nu(\alt(7))=\tfrac{1}{210},\
\tilde\nu(\alt(8))=\tfrac{19}{9720}, \
\tilde\nu(\alt(9))=\tfrac{1}{2160}.$$
Now suppose $n\geq 10$ and, for $1\leq i\leq n,$ let $$X_i=
\{s\in \perm(n)\mid s(i)=i\}\cong \perm(n-1),\quad Y_i=X_i\cap \alt(n)\cong \alt(n-1).$$
Consider $a_1,a_2\in \perm(n)$ and let $Y_{a_1,a_2}=\langle a_1,a_2,\alt(n)\rangle.$ Moreover, let
$$\Omega_0=\{(s_1,s_2)\in (\alt(n))^2\mid \langle s_1a_1,s_2a_2\rangle=Y_{a_1,a_2}\},$$
and,  for $1\leq i\leq n$, let 
$$\Omega_i=\{(s_1,s_2)\in (\alt(n))^2\mid  \langle s_1a_1,s_2a_2\rangle = Y_{a_1,a_2}\cap X_i\}.$$
Notice that $\Omega_0, \Omega_1,\dots,\Omega_n$ are disjoint subsets of $(\alt(n))^2$ and
$$\Omega_0\cup \Omega_1\cup \dots \cup \Omega_n\subseteq (\alt(n))^2\setminus \mathcal N_{a_1,a_2}(\perm(n),\alt(n)).$$ It follows that
$$\nu_{a_1,a_2}(\perm(n),\alt(n))\leq 1-\sum_{0\leq i\leq n}\frac{|\Omega_i|}{|\alt(n)|^2}.
$$
For $1\leq i\leq n,$ there exists $s_1, s_2\in \alt(n)$ such that $\langle s_1a_1, s_2a_2\rangle \leq X_i$ and
$$|\Omega_i|=|Y_i|^2\pi_{s_1a_1,s_2a_2}(X_i,Y_i)\geq |\alt(n-1)|^2\tilde\pi(\alt(n-1)).$$
Moreover $$|\Omega_0|\geq |\alt(n)|^2\tilde\pi(\alt(n)).$$
Hence $$\begin{aligned}\nu_{a_1,a_2}(\perm(n),\alt(n))&\leq 
1-\tilde\pi(\alt(n))-n\left(\frac{|\alt(n-1)|^2\tilde\pi(\alt(n-1))}{|\alt(n)|^2}\right)\\&=1-\tilde \pi(\alt(n))-\frac{\tilde\pi(\alt(n-1))}{n}.\end{aligned}$$
It follows from  \cite{nina} that if $m\geq 9$ then $\tilde\pi(\alt(m))\geq \tilde\pi(\alt(9))=\frac{15403}{18144}$ and therefore
for, $n\geq 10,$
$$\begin{aligned}\tilde\nu(\alt(n))&\leq 1-\tilde \pi(\alt(n))-\frac{\tilde\pi(\alt(n-1))}{n}\\&\leq 1-\frac{15403}{18144}\left(1+\frac{1}{10}\right)=\frac{12007}{181440}\approx 0,0622.\qedhere
\end{aligned}
	$$
\end{proof}
\begin{proof}[Proof of Theorem \ref{main}]
Assume that $G$ is a nonsolvable finite group. Then $G$ admits a nonabelian chief factor $X/Y$. In particular $G/C_G(X/Y)$ has a unique minimal normal subgroup, and this minimal normal subgroup is isomorphic to $X/Y.$ It can be easily proved that $\nu(G)\leq \nu(G/N)$ for every normal subgroup $N$ of $G$ (see for example \cite[Lemma 2.1]{cgk}). This implies that $\nu(G)\leq \nu(G/C_G(X/Y)).$ Moreover, by Proposition \ref{mono}, $\nu(G/C_G(X/Y))\leq \tilde\nu(S),$ where $S$ is a composition factor of $X/Y.$ Since, by Proposition \ref{simple}, $\tilde \nu(S)\leq 1/12$, we conclude
$$\nu(G)\leq \nu(G/C_G(X/Y))\leq \tilde \nu(S)\leq \frac{1}{12}.\qedhere
$$
\end{proof}
\section{Details on the computation of $\tilde \nu(S)$}

To compute the exact value of $\tilde{\nu}(S)$, with $S \in \mathcal{S}$, the following observation is useful.

\begin{lemma}
	Let $N$ be a normal subgroup of a finite group $G$, and let $\Omega$ be the set of pairs $(g_1, g_2) \in G^2$ such that $\langle g_1, g_2 \rangle N = G$. Then the value of $\nu_{g_1, g_2}(G, N)$ is the same for every choice of $(g_1, g_2) \in \Omega$.
\end{lemma}
\begin{proof}
 Denote by $\mathcal H$ the set of  2-generated subgroups $H$ of $G$ that are nilpotent and such that $G=HN$. Assume that $(g_1,g_2)\in \Omega$ and,
for every $H\in \mathcal H$, let $\Delta(g_1,g_2,H)$ be the set of $(n_1,n_2)$ such that $H=\langle n_1g_1,n_2g_2\rangle.$
Then $\mathcal N_{g_1,g_2}(G,N)$ is the disjoint union of the subsets $\Delta(g_1,g_2,H)$, with $H$ ranging  over $\mathcal H.$

Fix $H\in \mathcal H$. Since $HN=G,$ there exist $n_1,n_2\in N$ such that $g_1=n_1h_1$ and $g_2=n_2h_2$ and therefore $(n_1^{-1},n_2^{-1})\in \Delta(g_1,g_2,H)$. Moreover
$$\Delta(g_1,g_2,H)=\{(m_1,m_2)\in (N\cap H)^2\mid \langle (m_1n_1)(n_1^{-1}g_1),(m_2n_2)(n_2^{-1}g_2)\rangle=H\}.$$
It follows from the main result in \cite{gaz} that
$|\Delta(g_1,g_2,H)|/|N\cap H|^2$ does not depend on the choice of $(g_1,g_2)$ and coincides with the conditional probability that two randomly chosen elements of $H$ generate $H$ given that they generate $H$ modulo $H\cap N.$ Consequently, the cardinality of the sets $\Delta(g_1,g_2,H)$, and hence also that of their disjoint union, does not depend on the choice of $(g_1, g_2)$.
\end{proof}
Given $S \in \mathcal{S}$, for every almost simple group $G$ with $\mathrm{soc}(G) = S$, we use the table of marks of $G$ to determine the nilpotent 2-generated subgroups of $G$ that supplement $\mathrm{soc}(G)$, and for each such subgroup we compute the number of ordered pairs that generate it. Let $\tau(G, S)$ be the sum of the numbers thus obtained, divided by $|S|^2$ and by the number of pairs that generate $G/S$. The resulting number coincides with $\nu_{g_1, g_2}(G, S)$, where $\langle g_1, g_2 \rangle S = G$. Finally, $\tilde{\nu}(S)$ is the maximum of the values $\tau(G, S)$ as $G$ ranges over the almost simple groups with socle $S$.

The case which requires more attention is when $S=\psl(3,4)$. In this case, $\mathrm{Out}S\cong C_2\times \perm(3)$ and there are 6 different subgroups $S\leq G\leq \aut S$ such that $G/S$ is nilpotent and the values of $\tau(G,S)$ are as in the following table.

\begin{table}[h]
	\centering
	\renewcommand{\arraystretch}{1.4}
	\begin{tabular}{ccccccc}
		\toprule
		\textbf{$G$} & $S$ & $S.{2_1}$ & $S.{2_2}$ & $S.{2_3}$ & $S.2^2$ & $S.6$ \\
		\midrule
		\textbf{$\tau(G,S)$} & $\tfrac{5}{4032}$ & $\tfrac{13}{4032}$ & $\tfrac{19}{6720}$ & $\tfrac{5}{4032}$ & $\tfrac{13}{4032}$ & $\tfrac{1}{2520}$ \\
		\bottomrule
	\end{tabular}
	\caption{$\tau(G,S)$ when $\soc(G)=\psl(3,4)$}
\end{table}

\noindent In particular $\tilde\nu(\psl(3,4))=13/4032.$

\end{document}